\documentclass[10pt]{amsart}
\usepackage[american]{babel}
\usepackage{amssymb,mathrsfs,nicefrac}

\usepackage{hyperref}
\hypersetup{colorlinks=true, urlcolor=black, linkcolor=black, citecolor=black}

\newcommand{\AG}{A(G)}
\newcommand{\Soc}{\operatorname{Soc}}

\newcommand{\sd}{\triangle}
\newcommand{\NN}{\mathbb{N}}
\newcommand{\ZZ}{\mathbb{Z}}
\newcommand{\K}{\mathcal{K}}
\newcommand{\T}{\mathcal{T}}
\newcommand{\EE}{\mathscr{E}}
\newcommand{\bt}{\mathbf{t}}
\newcommand{\mm}{\mathfrak{m}}
\newcommand{\p}{\mathrm{p}}
\newcommand{\set}[1]{\left \{ #1 \right \}}

\newcommand{\ts}{\textstyle}

\newtheorem{theorem}{Theorem}[section]
\newtheorem{lemma}[theorem]{Lemma}
\newtheorem{prop}[theorem]{Proposition}
\newtheorem{cor}[theorem]{Corollary}

\theoremstyle{definition}
\newtheorem{definition}[theorem]{Definition}
\newtheorem{notation}[theorem]{Notation}
\newtheorem{rmk}[theorem]{Remark}
\newtheorem{rmks}[theorem]{Remarks}
\newtheorem{example}[theorem]{Example}

\begin{document}

\title[On the socle of Artinian algebras associated to graphs]%
{On the socle of Artinian\\ algebras associated to graphs}

\author{Jorge Neves}
\address{Univ.~Coimbra, CMUC, Department of Mathematics, 3001-501 Coimbra, Portugal.}
\email{neves@mat.uc.pt}

\thanks{This work was partially supported by the Centre for Mathematics of the University of Coimbra - UIDB/00324/2020,
funded by the Portuguese Government through FCT/MCTES. The author uses the computer algebra system Macaulay2, \cite{M2}, 
in the computations of examples.}

\makeatletter
\@namedef{subjclassname@2020}{%
  \textup{2020} Mathematics Subject Classification}
\makeatother


\begin{abstract}
Given a simple graph, consider the polynomial ring with coefficients in a field 
and variables identified with the edges of the graph. 
Given a non-empty even cardinality Eulerian subgraph and a choice of half of its edges, 
consider the homogeneous binomial obtained by taking the product of these 
edges minus the product of the remaining edges of the subgraph. 
We define a homogeneous ideal by taking as generators all binomials obtained in this way, 
varying the Eulerian subgraph and the choice of half of its edges, 
together with the squares of the variables of the ring. This ideal is related to 
the Eulerian ideal, introduced by Neves, Vaz Pinto and Villarreal. We call the corresponding quotient 
the Eulerian Artinian algebra associated to the graph. The goal of the present work is to study the socle
of these algebras through the lens of graph theory. Our main results include a combinatorial 
characterization of a monomial basis of the socle, a characterization of Gorenstein Eulerian Artinian 
algebras in the case of bipartite graphs and the computation of the $h$-vector and socle degrees in the cases 
of a complete graph and a complete bipartite graph. 
\end{abstract}
\maketitle

\section{Introduction}

Let $R=K[t_1,\dots,t_s]$ be a standard graded polynomial ring over a field $K$. 
A standard Artinian graded algebra, $A=R/I$, is a $0$-dimensional quotient of $R$ by a homogeneous ideal 
$I\subseteq R$. 
These graded rings have been studied from many points of view; one example 
is the study of their Lefschetz properties, especially in the Gorenstein and in the level case 
(\emph{cf.~\cite{Na22,Go17,GoZa18,HaWaWa19,McCuSe20},} 
for a small sample of recent works and \cite{MiNa13}, for a survey).
In this article, we study the socle of standard Artinian graded algebras associated to simple graphs.
By identifying the set of edges of a simple graph with the set of variables of $R$, 
we consider $I$ of the form:
$$
I = (\EE, t_1^2,\dots,t_s^2), 
$$
where $\EE$ is a set of binomials obtained from the even cardinality Eulerian subsets of edges of the graph. 
These standard Artinian graded algebras, which we will call \emph{Eulerian Artinian algebras}, 
are canonically related to the Eulerian ideal of the graph, introduced 
in \cite{NeVPVi20}. 
Our aim is to give a combinatorial characterization of the socle of an 
Eulerian Artinian algebra.

\medskip

In section~\ref{sec: prelim}, building upon the results of \cite{NeVa22} and \cite{Ne23} for the Eulerian ideal, 
we obtain a combinatorial characterization of a monomial basis for an Eulerian Artinian algebra, in terms 
of the graph (\emph{cf}.~Proposition~\ref{prop: K-basis in terms of reduced parity joins}). 
The monomials of this basis are square-free monomials that identify sets of edges of the graph 
containing no more than half the edges of any even cardinality Eulerian subgraph, 
subject to an additional condition coming from an 
ordering of the edge set. These distinguished sets of edges are called \emph{reduced parity joins} 
(\emph{cf}.~Definition~\ref{def: reduced parity join}) and they form a set  in bijection with the set 
of pairs $(T,p)$, where $T$, a set of vertices of the graph, and $p$, an element $\ZZ_2$, are 
such that there exists a $T$-join of cardinality of parity $p$ 
(\emph{cf}.~Proposition~\ref{prop: bijection reduced parity joins and pairs} and 
Corollary~\ref{cor: Hilbert series and reduced parity joins}).
In Sections~\ref{sec: complete} and \ref{sec: complete bipartite}, 
using these characterizations, we compute the \mbox{$h$-vector} of an 
Eulerian Artinian algebra for complete and complete bipartite graphs. 
In section~\ref{sec: socle}, we study the socle of an Eulerian Artinian algebra. 
We start by showing that it is generated over $K$ by monomial residues 
(\emph{cf}.~Proposition~\ref{thm: Soc of J is generated by monomials}) and then give a combinatorial 
characterization of the sets of edges that the monomials in the socle basis identify 
(\emph{cf}.~Theorem~\ref{thm: characterization of socle}).
As an application, we show that if the graph is bipartite, 
the only Gorenstein Eulerian Artinian algebras are the complete intersection ones, 
that are obtained when $\EE$, above, is empty (\emph{cf}.~Theorem~\ref{thm: Gorenstein for bipartite}).

\section{Preliminaries}\label{sec: prelim}

\subsection{Setup}
Throughout, $G$ will denote a simple, undirected graph. 
The vertex set of $G$ will be denoted by $V_G$ and the edge set, which we regard as a non-empty set 
of subsets of $V_G$, will be denoted by $E_G$. We associate to $G$ the polynomial ring 
$$K[E_G] = K[t_e : e\in E_G],$$ 
where $K$ is any field. If $\alpha \in \NN^{E_G}$, the monomial given by the product 
of $t_e^{\alpha(e)}$, when  $e$ varies in $E_G$, will be denoted by $\bt^\alpha$. 
If $J\subseteq E_G$ is any set of edges, we will also use 
$\bt_J$ for the product of all variables identifying edges in $J$.
\medskip 

Let $C\subseteq E_G$ and $v\in V_G$. The degree of $v$ in $C$ is defined by 
$$
\ts \deg_C(v) = \sum\limits_{e\in C} |\set{v}\cap e\,|.
$$
A subset, $C$, is called Eulerian if $\deg_C(v)$ is even, for every $v\in V_G$. 
Let $J,L\subseteq E_G$ be non-empty. The binomial $\bt_J-\bt_L$ is called \emph{Eulerian} 
if $J\cap L = \emptyset$, $|J|=|L|$ and $C=J\cup L$ is Eulerian. 
Let us denote the set of Eulerian binomials by $\EE$.

\begin{definition}\label{def: saga of G}
If $G$ is any simple graph, define $I(G) = (\EE \cup \set{t_e^2 : e \in E_G})$ 
and define the \emph{Eulerian Artinian algebra} associated to $G$ by $\AG = K[E_G]/I(G)$.
\end{definition}

If the only Eulerian subset of edges of even cardinality is the empty set, which 
is the case when $G$ is a forest or a graph with a unique odd cycle, then  
$$
\AG=K[E_G]/(t_e^2 : e\in E_G),
$$ 
which is a well-known complete intersection. 

\subsection{Eulerian ideal}\label{subsec: Eulerian ideal and Artinian algebra}
Consider 
$K[V_G]=K[x_i : i\in V_G]$ and let $\varphi \colon K[E_G] \to K[V_G]$ be the ring homomorphism 
defined by $\varphi(t_e) = x_ix_j$ for every $e=\set{i,j}\in E_G$.
The Eulerian ideal of $G$ was defined in \cite[Definition~1.1]{NeVPVi20} as
$$
I(X_G) = \varphi^{-1} (x_i^2-x_j^2 : i,j\in V_G).
$$
It follows from \cite[Proposition~2.5]{NeVPVi20} that  
$I(X_G) = (\EE\cup \{t_e^2-t_f^2 : e,f\in E_G\})$. In other words,
$I(G)$ and $I(X_G)$ are related by 
$I(G) = I(X_G) + (t_e^2)$, for any choice of $e\in E_G$.

\subsection{A Gr\"obner basis}
In the approach of \cite{Ne23,NeVa22}, the first step in the study of the Eulerian ideal 
$I(X_G)$ was to describe a Gr\"obner basis. As we explain next, 
a Gr\"obner basis for $I(G)$ can be easily derived from 
the results of \cite{Ne23,NeVa22}. 

\begin{definition}\label{def: monomial order}
Given an ordering of the edge set $E_G=\{e_1,\dots, e_s\}$, define $\succeq$, the 
associated graded reversed lexicographic order on $K[E_G]$, by postulating that 
$\bt^\alpha \succ \bt^\beta$, where $\alpha\not =\beta \in \NN^{E_G}$, if and only if $\deg(\bt^\alpha) > \deg(\bt^\beta)$
or $\deg(\bt^\alpha)=\deg(\bt^\beta)$ and the last nonzero coordinate of 
$(\alpha(e_1)-\beta(e_1),\dots,\alpha(e_s)-\beta(e_s))$ is negative.
\end{definition}

\begin{prop}\label{prop: Grobner basis of IG}
Given an ordering of $E_G$, the set  
$\EE \cup \{t_e^2 : e \in E_G\}$ is a Gr\"obner basis for $I(G)$ with respect 
to the associated grevlex order. 
\end{prop}

\begin{proof}
By \cite[Theorem 3.3]{Ne23} the set $\EE \cup \{t_e^2-t^2_f : e,f \in E_G\}$ is a Gr\"obner basis for 
the Eulerian ideal $I(X_G)$, with respect to the grevlex order. 
Consider $\ell \in E_G$, the last edge. Then, as no initial term of an element of this Gr\"obner basis 
is divisible by $t_\ell$, by Buchberger's criterion, we deduce that 
\begin{equation}\label{eq: L233}
\EE \cup \{t_e^2-t^2_f : e,f \in E_G\} \cup \set{t_\ell^2} 
\end{equation}
is a Gr\"obner basis for $I(G)=I(X_G)+(t_\ell^2)$. Since   
the initial terms of members of the set $\EE \cup \{t_e^2 : e\in E_G\}$ generate the same 
ideal as those of \eqref{eq: L233}, this
set is also Gr\"obner basis for $I(G)$.
\end{proof}

Let us denote the symmetric difference of sets by $\sd$.

\begin{cor}\label{cor: membership of square-free monomials}
Let $M\subseteq E_G$ be any set of edges. Then $\bt_M \in I(G)$ if and only if 
there exists an even cardinality Eulerian subset $C\subseteq E_G$ such that 
$|M\sd C|<|M|$.
\end{cor}

\begin{proof}
Suppose that $\bt_M$ is divisible by the leading term 
of an element of $\EE$. Without loss of generality, 
let this be $\bt_J-\bt_L$, with $J\cap L = \emptyset$, $|J|=|L|$ and 
$C=J\cup L$ an Eulerian subset, with last edge in $L$. 
Then $\bt_J\mid \bt_M\iff J\subseteq M$ and  
\begin{equation}\label{eq: L289}
\bt_M = \bt_{M\setminus J}(\bt_J-\bt_L) + \bt_{M\setminus J}\bt_{L}.
\end{equation}
If $\bt_{M\setminus J}\bt_{L}$ is a square-free monomial then
$$
M\cap C = J\iff (M\setminus J)\cup L = M\sd C,
$$
so that $\bt_{M\setminus J}\bt_{L} = \bt_{M\sd C}$ and 
$|M|=|M\sd C|$. If $\bt_{M\setminus J}\bt_{L}$ is not square-free then 
$$
|M\cap C|>|J| = |C|/2  \iff |M\sd C| < |M|.
$$

If there exists an even cardinality Eulerian subset $C\subseteq E_G$ with 
$$
|M\sd C|<|M| \iff |M\cap C|>|C|/2,
$$ 
we may choose $J\subseteq M\cap C$, a set of cardinality $|C|/2$ that 
does not include the last edge of $C$. Then $\bt_J-\bt_L$ is an Eulerian binomial as above and in \eqref{eq: L289}
the monomial $\bt_{M\setminus J}\bt_{L}$ is divisible by a square of a variable. Therefore $\bt_M \in I(G)$.
Conversely, if $\bt_M\in I(G)$, then $\bt_M$ has remainder zero in its standard expression with respect 
to the Gr\"obner basis of Proposition~\ref{prop: Grobner basis of IG}. Let us compute 
this remainder by first dividing only by the elements of $\EE$ until we get a monomial that is not square-free. 
Let $C_1,\dots,C_m$, with $m\geq 1$ be the corresponding even cardinality Eulerian subsets. By our initial argument,  
$$
\ts |(M\sd C_1\sd \cdots \sd C_{m-1})\sd C_m| < |M\sd C_1\sd \cdots \sd C_{m-1}|=|M|.
$$
Let $C=C_1\sd C_2 \sd \cdots \sd C_m$, which is also an even cardinality Eulerian subset. Then 
$|M\sd C|<|M|$.
\end{proof}

Subsets $M\subseteq E_G$ for which the condition in Corollary~\ref{cor: membership of square-free monomials}
is not satisfied are called parity joins. 

\begin{definition}[{\emph{cf}.~\cite[Definition 4.11]{Ne23}}]\label{def: parity join}
A subset $J\subseteq E_G$ is called a parity join if, 
for every Eulerian subset $C\subseteq E_G$ of even cardinality,
$$
\ts |J\cap C|\leq  \frac{|C|}{2}\iff |J\sd C|\geq|J|.
$$
\end{definition}

It follows from Corollary~\ref{cor: membership of square-free monomials} 
that the maximum of $d$, for which $\AG_d$ is nonzero, which is called the 
\emph{maximum socle degree} of $\AG$, is the maximum cardinality of a parity join of the graph.

\subsection{A monomial basis}
If $I$ is an ideal in a polynomial ring endowed with a monomial order then, 
by Macaulay's theorem (\emph{cf}.~\cite[Theorem~2.6]{EnHe12}),
the set of residues of monomials that do not belong to the initial ideal of $I$ is a $K$-basis 
for the quotient of the polynomial ring by the ideal. In our case, fixing any ordering of $E_G$, 
a monomial not in the initial ideal of $I(G)$ is necessarily of the form $\bt_J$
for some parity join $J\subseteq E_G$, but not conversely. 
The following notion was introduced in \cite[Definition 14]{NeVa22}.

\begin{definition}\label{def: reduced parity join}
Given an ordering of $E_G$, a parity join $J\subseteq E_G$ is said \emph{reduced}
if, whenever $C\subseteq E_G$ is a non-empty, even cardinality Eulerian subset and 
$$
|J\sd C|=|J| \iff |J\cap C| = |C|/2,
$$
then $J$ contains the last edge in $C$.
\end{definition}

\begin{prop}\label{prop: K-basis in terms of reduced parity joins}
Given any ordering of $E_G$, 
the set of monomial residues given by $\bt_J + I(G)$, with $J$ in the set of reduced parity joins of $G$, 
is a $K$-basis of $\AG$.
\end{prop}

\begin{proof}
See \cite[Theorem 15]{NeVa22}.
\end{proof}

One can give a second characterization this $K$-basis of $\AG$ 
using the related notion of $T$-join of a graph.  
Let $T\subseteq V_G$ and $J\subseteq E_G$. We say that $J$ is a $T$-join if 
and only if 
$$
T=\{v\in V_G : \deg_J(v) \text{ is odd}\}.
$$
Accordingly, a subset $C\subseteq E_G$ is Eulerian if and only if $C$ is
an $\emptyset$-join.
Given $T\subseteq V_G$, 
one can show that there exists a $T$-join if and only if the cardinality of the intersection
of $T$ with the vertex set of every connected component of $G$ is even 
(\emph{cf}.~\cite[Proposition~12.7]{KoVy18}). 

\begin{notation}
Let $\T_G$ denote the set of $T\subseteq V_G$ for which there exists a $T$-join.
\end{notation} 

\noindent
Fix $T\in \T_G$ and let $J_1$ and $J_2$ be two $T$-joins. Then
$J_1\sd J_2$, is an $\emptyset$-join (\emph{cf}., for example, \cite[Proposition~12.6]{KoVy18}); 
in other words, $J_1\sd J_2$ is an Eulerian subset. Since 
$$
|J_1\sd J_2| \equiv_2 |J_1| + |J_2|,
$$
if there exist $T$-joins of cardinalities of opposite parity, then the graph contains an odd cycle, 
i.e., it is non-bipartite. The converse is also true (\emph{cf}., for example, \cite[Lemma~4.6]{Ne23}).

\begin{definition}\label{def: T p joins}
Let $T\subseteq V_G$, $p\in \ZZ_2$ and $J\subseteq E_G$.
\begin{enumerate}
\item $J$ is called a $(T,p)$-join if $J$ is a $T$-join and $|J|+2\ZZ = p$.
\item Let $\T_{\p,G}$ denote the set of $(T,p)\in \T_G\times \ZZ_2$
admitting a $(T,p)$-join.
\item If $(T,p)\in \T_{\p,G}$, denote $\tau(T,p) = \min \{|J| : \text{$J$ is a $(T,p)$-join} \}$.
\item Define $\theta(J) = (T,p)\in \T_{\p,G}$, by setting 
$T=\{v\in V_G : \deg_J(v) \text{ is odd}\}$ and $p = |J|+2\ZZ$. 
\end{enumerate}
\end{definition}

If $G$ is non-bipartite, then $\T_{\p,G} = \T_G\times \ZZ_2$, while 
if $G$ is bipartite, then, for each $T\in \T_G$, the set $\T_{\p,G}$ contains exactly one of the pairs 
$(T,0)$ or $(T,1)$ and thus the cardinality of a $T$-join is then determined by $T$. 
For a bipartite $G$, we abbreviate $\tau(T,p)$ to $\tau(T)$.
\medskip 

The following lemma follows from the elementary property 
of $T$-joins used above (\emph{cf}.~\cite[Proposition~12.6]{KoVy18}) 
and the formula for the cardinality of a symmetric difference.

\begin{lemma}\label{lem: symetric difference of Ti-joins}
If $J_1$ is a $(T_1,p_1)$-join and $J_2$ is a $(T_2,p_2)$-join, then  
$J_1\sd J_2$ is a $(T_1\sd T_2,p_1+p_2)$-join. In particular, if $J_1$ and $J_2$ are 
two $(T,p)$ joins, then there exists and even cardinality Eulerian subset $C\subseteq E_G$ 
such that $J_2=J_1\sd C$. 
\end{lemma}

\begin{rmk}\label{rmk: parity joins as minimum cardinality T-joins}
Let $J\subseteq E_G$ and let $(T,p)=\theta(J)$. 
The lemma above can be used to show that 
$J$ is a parity join if and only if $\tau(T,p)=|J|$, i.e., 
parity joins are the same as minimum cardinality $(T,p)$-joins.
\end{rmk}

Finally, the next result provides a second combinatorial characterization of the monomial basis of $\AG$
of Proposition~\ref{prop: K-basis in terms of reduced parity joins}, in terms of $T$-joins.

\begin{prop}\label{prop: bijection reduced parity joins and pairs}
Fix an ordering of $E_G$. For every $(T,p)\in \T_{\p,G}$,
there exists a unique reduced parity join among all $(T,p)$-joins. 
\end{prop}

\begin{proof}
See \cite[Proposition~22]{NeVa22}.
\end{proof}

It follows that the unique reduced parity join in the set of $(T,p)$-joins has  
cardinality equal to $\tau(T,p)$ and the restriction of $\theta$ to the set 
reduced parity joins of $G$ is a bijection with $\T_{\p,G}$. Using this and 
Proposition~\ref{prop: K-basis in terms of reduced parity joins}, we obtain the following corollary. 

\begin{cor}\label{cor: Hilbert series and reduced parity joins}
The Hilbert series of $\AG$, in the variable $z$, is the sum of all $z^{\tau(T,p)}$
as $(T,p)$ varies in $\T_{\p,G}$ and, fixing any ordering of the set of edges, it 
is the sum of all $z^{|J|}$ where $J$ varies in the set of reduced parity joins of $G$.
\end{cor}

Let us end this section by using Corollary~\ref{cor: Hilbert series and reduced parity joins} 
to compute the $h$-vector of $\AG$ for graphs admitting a single non-empty, even cardinality 
Eulerian subset of edges. We will also use this result, in Sections~\ref{sec: complete} 
and \ref{sec: complete bipartite}, to compute the $h$-vector of $\AG$ for 
a complete graph and a complete bipartite graph, respectively.

\begin{prop}
Suppose that $E_G$ possesses a single non-empty, even car\-di\-na\-li\-ty Eulerian subset of edges, $C\subseteq E_G$.
Let us denote $s=|E_G|$, $c=\nicefrac{|C|}{2}$ and let $(h_0,h_1,\dots,h_r)$ denote the $h$-vector of $\AG$.
Then
$$
\renewcommand{\arraystretch}{1.4}
h_i= 
\left\{
\begin{array}{l}
\binom{s}{i}, \quad \text{if}\quad  0\leq i\leq c-1,\\
\binom{2c-1}{c-1}\binom{s-2c}{i-c}+\sum_{k\geq 1} \binom{2c}{c-k}\binom{s-2c}{i-c+k}, \quad \text{if} \quad i\geq c.\\
\end{array}
\right.
$$
In particular, $r$, the socle degree of $\AG$, is equal to $s-c$.
\end{prop}

\begin{proof}
If $J\subseteq E_G$ is such that $|J|\leq c-1$ then, necessarily, 
$|J\cap C|< c$ and therefore, under our assumption on the graph, $J$ is a reduced parity join. Hence,
$h_i = \binom{s}{i}$, for all $0\leq i \leq c-1$. If $i\geq c$, then any reduced parity join 
must have at most $c$ edges in $C$ and if this number is maximum, then it must include the last edge 
in $C$. Hence, for $i\geq c$,
$$
\ts h_i = \binom{2c-1}{c-1}\binom{s-2c}{i-c}+\sum_{k\geq 1} \binom{2c}{c-k}\binom{s-2c}{i-c+k}.\qedhere
$$
\end{proof}


\section{The socle of $\AG$}\label{sec: socle}
Let $R$ be a polynomial ring with coefficients in a field $K$, standard graded, and
denote by $\mm$ its irrelevant ideal. If $M$ is a graded module over $R$, the socle of 
$M$ is the graded submodule $\Soc M = \{ u\in M : \mm u  = (0) \}$. The 
degrees of the minimal generators of the socle of a standard Artinian graded algebra determine the last 
term of the minimal graded free resolution of the algebra over the polynomial ring 
(\emph{cf}.~\cite[Lemma~5.3.3]{monalg}). 
In particular, a standard Artinian graded algebra is 
Gorenstein if and only if its socle is $1$-dimensional over $K$. 
The degrees of the minimal generators of the socle module are called the 
\emph{socle degrees}; the dimension of the socle, is called the \emph{type}. A standard Artinian 
graded algebra is called \emph{level} if its socle is generated in a single degree.

\begin{theorem}\label{thm: Soc of J is generated by monomials}
The socle of $\AG$ is generated by residues of monomials. 
\end{theorem}

\begin{proof}
Choose an ordering of $E_G$ and consider the associated 
grevlex order on $K[E_G]$. Let $0\not =u\in \Soc \AG$. 
By Proposition~\ref{prop: K-basis in terms of reduced parity joins}, we may write 
$$
u = (c_1\bt_{J_1} + \cdots + c_m\bt_{J_m})+I(G)
$$
for some $J_1,\dots,J_m$, reduced parity joins, $c_i\in K\setminus \{0\}$ and $m>0$.
By assumption, for every $e\in E_G$,
\begin{equation}\label{eq: L531}
c_1\bt_{J_1}t_e + \cdots + c_m\bt_{J_m}t_e \in I(G).
\end{equation}
Fix $e\in E_G$ and, to ease notation, assume that there exists $0\leq k\leq m$ such that  
$c_i\bt_{J_i}t_e \not \in I(G)$, for all $1\leq i\leq k$ and 
$c_i\bt_{J_i}t_e\in I(G)$, for all $k+1\leq i\leq m$. We want to show that $k=0$.
Given that $\EE$ consists of binomials, the reduction of a monomial modulo this set 
is also a monomial. Accordingly, let $d_i\bt^{\beta_i}$, with $d_i\not = 0$, 
denote the reduction of $c_i\bt_{J_i}t_e$ modulo $\EE$.  
Since $\EE \cup \{t_e^2 : e \in E_G\}$ is a Gr\"obner basis for $I(G)$, 
by our assumptions, none of $d_1\bt^{\beta_1},\dots,d_k\bt^{\beta_k}$
is divisible by a square, while all of $d_{k+1}\bt^{\beta_{k+1}},\dots,d_m\bt^{\beta_m}$ are.
Hence, the remainder of $c_1\bt_{J_1}t_e + \cdots + c_m\bt_{J_m}t_e$ modulo 
$\EE \cup \{t_e^2 : e \in E_G\}$ is $d_1\bt^{\beta_1}+\cdots+d_k\bt^{\beta_k}$.
From \eqref{eq: L531} it follows that  
\begin{equation}\label{eq: L568}
d_1\bt^{\beta_1}+\cdots+d_k\bt^{\beta_k} = 0.
\end{equation}
By \cite[Theorem 3.3]{Ne23},  $\EE \cup \{t_e^2-t^2_f : e,f \in E_G\}$ is a Gr\"obner basis for 
the Eulerian ideal, $I(X_G)\subseteq K[E_G]$, and hence, using \eqref{eq: L568}, we deduce that 
$$
(c_1\bt_{J_1} + \cdots + c_k\bt_{J_k})t_e \in I(X_G).
$$
Since $t_e$ is regular on $K[E_G]/I(X_G)$ 
(\emph{cf}.~\cite[Proposition~2.1]{NeVa22}), we get
$$
c_1\bt_{J_1} + \cdots + c_k\bt_{J_k} \in I(X_G)\subseteq I(G),
$$ 
which is a contradiction. We conclude that $k=0$. Varying $e\in E_G$, we 
deduce that $\bt_{J_i} + I(G)$ belongs to $\Soc \AG$, 
for every $i=1,\dots,r$.
\end{proof}

Let us now address the problem of giving a combinatorial characterization of 
reduced parity joins yielding monomial residues in the socle of $\AG$.
Fix an ordering of $E_G$. If $J$ is a reduced parity join and $\bt_J+I(G)$ belongs to the socle 
of $\AG$, then $\bt_J t_e \in I(G)$, for all $e\in E_G$. 
Using Corollary~\ref{cor: membership of square-free monomials}, 
what this means is that $J\cup \set{e}$ is no longer a parity join, for every 
$e\not \in J$. 
Hence a monomial socle basis for $\AG$ is in bijection with the set of reduced 
parity joins that are maximal, under inclusion, in the set of \emph{all}
parity joins. Note that, as Example~\ref{exa: maximal reduced parity join} shows, 
not all reduced parity joins that are maximal in the set of reduced parity joins 
are maximal in the set of all parity joins. 
\smallskip 

This combinatorial characterization of the socle depends on the choice of ordering of $E_G$.
Next, we shall give another combinatorial characterization 
in terms of the invariants $\tau(T,p)$ 
of the graph, which, in particular, is independent of an ordering of $E_G$. 
Before, we need to state an elementary result, concerning these numbers. 
Recall that we are regarding an edge of $G$ as a subset of $V_G$. 
Hence, if $e\in E_G$ and $T\subseteq V_G$ it makes sense to take the symmetric difference $T\sd e$. 

\begin{lemma}\label{lem: effect of adding edge}
Let $(T,p)\in \T_{\p,G}$ and $e\in E_G$. Then $(T\sd e,p+1)\in \T_{\p,G}$ and 
$$
\tau(T\sd e, p+1) = \tau(T,p)\pm 1.
$$
\end{lemma}

\begin{proof}

If $J\subseteq E_G$ is a $(T,p)$-join, then, by Lemma~\ref{lem: symetric difference of Ti-joins}, 
$J\sd \set{e}$ is a $(T\sd e,p+1)$-join. Hence, $\tau(T\sd e,p+1) \leq \tau(J,p) +1$.
Repeating this argument with $(T\sd e,p+1)$, we deduce that
$$
\tau(T,p)-1 \leq \tau(T\sd e, p+1) \leq \tau(J,p)+1
$$
so that $\tau(T\sd e, p+1) = \tau(T,p)\pm 1$.
\end{proof}

\begin{definition}
If $(T,p)\in \T_{\p,G}$ satisfies $\tau(T\sd e,p+1)=\tau(T,p)-1$, for every $e\in E_G$, then 
we will refer to $(T,p)$ as a \emph{socle pair} of $G$. To ease notation, if $G$ is bipartite, we omit the parity, 
and we will refer to $T\in \T_G$ as a \emph{socle set} if and only if $\tau(T\sd e) = \tau(T)-1$, for every 
$e\in E_G$.
\end{definition}

We are now ready to state our second main result. 

\begin{theorem}\label{thm: characterization of socle}
The socle degrees of $\AG$ are given by $\tau(T,p)$, for $(T,p)$ varying 
in the set of socle pairs. Alternatively, fixing an ordering of $E_G$, the socle degrees 
are the cardinalities of reduced parity joins $J\subseteq E_G$ such that 
$\bt_J t_e \in I(G)$, for all $e\in E_G$, or, equivalently, the cardinalities of reduced parity joins that 
are maximal in the set of all parity joins.
\end{theorem}

\begin{proof}
Let us fix an ordering of $E_G$. Let $\set{J_1,\dots,J_t}$ be the set of reduced parity joins $J\subseteq E_G$ 
such that $\bt_J t_e \in I(G)$, for every $e\in E_G$.  By Proposition~\ref{prop: K-basis in terms of reduced parity joins} 
and Theorem~\ref{thm: Soc of J is generated by monomials}, 
$\set{\bt_{J_1}+I(G),\dots,\bt_{J_t}+I(G)}$ is a $K$-basis for the socle of $\AG$ and hence
the socle degrees are given by the cardinalities of $J_1,\dots,J_t$. 
Now, let $J$ be \emph{any} reduced parity join and denote $\theta(J)=(T,p)$.
Since $\tau(T,p)=|J|$ (\emph{cf}.~Remark~\ref{rmk: parity joins as minimum cardinality T-joins}), using 
Lemma~\ref{lem: effect of adding edge}, to prove the first assertion, 
it suffices to show that for any $e \in E_G$, 
\begin{equation}\label{eq: L548}
\bt_J t_e \in I(G)\iff \tau(T\sd e,p+1)< |J|+1.
\end{equation}
If $e\in J$ then $\bt_J t_e$ is divisible by a square, hence belongs to $I(G)$, and, 
on the other hand, $J\sd \{e\}$ is a $(T\sd e,p+1)$-join of cardinality $|J|-1$.
If $e\not \in J$ then, using Corollary~\ref{cor: membership of square-free monomials} and 
Lemma~\ref{lem: symetric difference of Ti-joins}, 
$\bt_Jt_e \in I(G)$ if and only if there exists a even cardinality Eulerian subset $C\subseteq E_G$
such that 
$$
\ts |C\sd(J\cup \set{e})|<|J\cup \set{e}| = |J|+1,
$$ 
which is equivalent to $\tau(T\sd e,p+1)<|J|+1$.
\end{proof}

\begin{cor}
If $E_G$ possesses a single non-empty, even cardinality Eulerian subset, $C\subseteq E_G$, 
then $\AG$ is level of socle degree $s-c$, and type $\binom{2c-1}{c-1}$, 
where $s=|E_G|$ and $c=\nicefrac{|C|}{2}$. 
\end{cor}

\begin{proof}
If $J\subseteq E_G$ is a reduced parity join which is maximal in the set of parity joins then 
$J\cap C = \frac{|C|}{2}$, for otherwise the last edge of $C\setminus J$ could be added to $J$
to form another reduced parity join. Similarly, $J$ must include all edges of $G\setminus C$.
Hence the cardinality of any such reduced parity join is $c+(s-2c)=s-c$, hence 
$\AG$ is level, and their number is $\binom{2c-1}{c-1}$.
\end{proof}

Let us finish this section by giving an application Theorem~\ref{thm: characterization of socle}.

\begin{theorem}\label{thm: Gorenstein for bipartite}
If $G$ is bipartite, then $\AG$ is Gorenstein if and only if 
$E_G$ does not possess any non-empty even cardinality Eulerian subset. 
If this is the case, then 
$I(G)=(t_e^2 : e \in E_G)$ 
and therefore $\AG$ is a complete intersection. 
\end{theorem}

\begin{proof}
If $E_G$ does not contain any non-empty, even cardinality Eulerian subset then 
$\EE=\emptyset$ and, accordingly, $I(G)=(t_e^2 : e \in E_G)$. We deduce that $\AG$
is then a complete intersection and therefore that 
$\AG$ is Gorenstein (\emph{cf}.~\cite[Corollary~A.6.4]{HeHi11}).
Let us assume that there exists a non-empty, even cardinality Eulerian subset of $E_G$.
Choose $C\subseteq E_G$, one such set of minimum cardinality, and let us show that the socle of $\AG$ is not $1$-dimensional.  
Since $G$ is bipartite, $C$ is the set of edges of a smallest 
length even cycle. Let us denote $2c=|C|\geq 4$ and let us fix an ordering of $E_G$ such that the last edge of $G$ belongs to $C$. 
Let $v$ and $v'$ be the vertices of the last edge. 
Let $J_1$ denote the set of edges of the path, contained in $C$, 
that starts at $v$, goes through $v'$ and has length $c$. 
Let $w$ denote the second end-point of this path.
Similarly, let $J_2$ denote the set of edges of the path, 
contained in $C$, that starts at $v'$, goes 
through $v$ and has length $c$, and let us denote by $w'$ the second 
end-point of this path.  
Given that $c$ is the minimum cardinality of an Eulerian subset,
any subset of $c$ elements of $C$ that includes the last edge
is a reduced parity join in $G$. Hence both $J_1$ and $J_2$ are
reduced parity joins. In particular, 
if $T_1=\{v,w\}$ and $T_2=\{v',w'\}$, 
then $\tau(T_1)=\tau(T_2)=c$ (\emph{cf}.~Remark~\ref{rmk: parity joins as minimum cardinality T-joins},
note also that we are omitting the parity in the notation, as
$G$ is bipartite). If $T_1$ is not a socle set then 
there exists $e\in E_G$ such that $\tau(T_1\sd e)=c+1$. Clearly,
$e$ cannot belong to $C$. The same applies to $T_2$.
Iterating this procedure, let $\{e_1,\dots,e_k\}$ and $\{f_1,\dots,f_r\}$ 
be sets of edges (possibly empty) of cardinality $k$ and 
$r$, respectively, such that both $\{e_1,\dots,e_k\}$ and $\{f_1,\dots,f_r\}$ have empty intersection with $C$, both 
$T_1\sd e_1\sd \dots \sd e_k$ and $T_2\sd f_1 \sd \dots \sd f_r$
are socle sets and such that 
$$
\tau(T_1\sd e_1\sd \dots \sd e_k) = c+k,\quad 
\tau(T_2\sd f_1 \sd \dots \sd f_r) = c+r.
$$
Note that since $J_1\cup \set{e_1,\dots,e_k}$ is a $(T_1\sd e_1 \sd \cdots \sd e_k)$-join of cardinality 
$$
c+k=\tau(T_1\sd e_1 \sd \cdots \sd e_k),
$$
we deduce that $J_1\cup \set{e_1,\dots,e_k}$
is a parity join (\emph{cf}.~Remark~\ref{rmk: parity joins as minimum cardinality T-joins}). 
The same applies to $J_2\cup \set{f_1,\dots,f_r}$, however either of them may not be reduced. 
If $k\not =r$ then these socle sets yield two socle monomials 
of different degrees, hence $\AG$ is not level and therefore the socle 
of $\AG$ is not $1$-dimensional. If $k=r$ and the socle 
of $\AG$ is $1$-dimensional then we must have 
\begin{equation}\label{eq: L643}
\renewcommand{\arraystretch}{1.3}
\begin{array}{c}
T_1\sd e_1 \sd \cdots \sd e_k = T_2\sd f_1\sd \cdots \sd f_k\\
\iff T_1\sd T_2 = \{v,v',w,w' \} = e_1 \sd \cdots \sd e_k\sd f_1\sd \cdots \sd f_k.
\end{array}
\end{equation}
With a view to a contradiction, assume this holds and,
from the equality in \eqref{eq: L643}, let us draw a contradiction. 
Without loss of generality, we may assume
$$\{e_1,\dots,e_k\}\cap \{f_1,\dots,f_k\} = \emptyset.$$ 
If 
$\{e_1,\dots,e_k\}\cup \{f_1,\dots,f_k\}$
contains the edge set of any cycle, then, because $G$ is bipartite, 
such cycle must be even. Let $2m$ be its length. 
Since, as we argued above, \mbox{$J_1\cup \set{e_1,\dots,e_k}$} 
and $J_2\cup \set{f_1,\dots,f_r}$
are parity joins, such cycle must consist of $m$ edges in 
$\set{e_1,\dots,e_k}$, say $e_{i_1},\dots,e_{i_m}$, and  $m$ edges in $\set{f_1,\dots,f_k}$, say 
$f_{j_1}, \dots, f_{j_m}$, in which event, 
$$
e_{i_1}\sd \cdots \sd e_{i_m} = f_{j_1}\sd \cdots \sd f_{j_m}.
$$
Thus, we may assume that $\{e_1,\dots,e_k\}\cup \{f_1,\dots,f_k\}$ 
does not contain the edge set of any cycle. Then \eqref{eq: L643} implies that 
$$
\{e_1,\dots,e_k\}\cup \{f_1,\dots,f_k\} = P\sqcup Q,
$$
where $P\sqcup Q$ denotes the disjoint union of the edge sets of two paths, 
which have empty intersection, and whose end-points form the set $\{v,w,v',w'\}$.
Let us denote 
$$
P_e = P\cap \{e_1,\dots,e_k\}, \quad P_f = P\cap \{f_1,\dots,f_k\}
$$
and, similarly, let us use the notation $Q_e$ and $Q_f$ for the corresponding subsets of $Q$. 
Without loss of generality, 
we may assume that $v$ in an end-point of $P$. 
There are three cases to consider. 
\smallskip 

\noindent
First, let us consider the case when the end-points of $P$ are $v$ and $w$. 
Consider the Eulerian subset given by $J_1\sd P$. Then, because 
$P\cap C = \emptyset$ and $J_1\subseteq C$, it follows that 
$J_1\sd P = J_1\sqcup P$ and hence $|J_1\sd P|=c+|P|>0$. Also, as $G$ is bipartite, we deduce that $c+|P|$ is even.
Since both $J_1\cup P_e\subseteq J_1\cup \set{e_1,\dots,e_k}$ and $J_2\cup P_f\subseteq J_2\cup \set{f_1,\dots,f_k}$
are parity joins. Using the condition of the definition (\emph{cf}.~Definition~\ref{def: parity join}) with 
$C=J_1\sd P$ we get
$$
\ts c+|P_e| \leq \frac{c+|P|}{2} \quad\text{and} \quad 1+|P_f|\leq \frac{c+|P|}{2}
$$
which implies $\ts c+1+|P_e|+|P_f| \leq c+|P| \iff 1\leq 0$, a contradiction. 
\medskip 

\noindent
Second, let the end-points of $P$ be $v$ and $w'$. 
Arguing as before, we deduce that  
$J_1\sd\{\{w,w'\}\} \sd P$ is an Eulerian subset of 
even cardinality, equal to $c+1+|P|$.
Since 
$J_1\cup P_e$ is a parity join, we get:
$$
\ts c+|P_e|\leq \frac{c+1+|P|}{2}. 
$$
Similarly, since $J_2\sd\{\{v,v' \}\}\sd P$, which has cardinality $c-1+|P|$, is also an Eulerian subset of even cardinality 
and $J_2\cup P_f$ is a parity join, we get 
$$
\ts c-1+|P_f|\leq \frac{c-1+|P|}{2}. 
$$
Adding the two inequalities above, we conclude that $c\leq 1$, which is a contradiction.
\smallskip

\noindent
Third and last, assume that the end-points of $P$ be $v$ and $v'$, so that 
$Q$ has end-points $w$ and $w'$. Consider the non-empty, even cardinality 
Eulerian subset given by $C\sd\{\{w,w'\}\}\sd Q$. Using the pa\-ri\-ty joins 
$J_1\cup Q_e$ and $J_2\cup Q_f$, in the same way as before, we deduce that 
$2c+|Q|\leq 2c-1+|Q|\iff 0\leq -1$.
\end{proof}

\begin{cor}
If $G$ is bipartite, then $K[E_G]/I(X_G)$, the quotient of $K[E_G]$ 
by the Eulerian ideal, is Gorenstein if and only if $E_G$ does not possess any non-empty 
even cardinality subset.
\end{cor}

\begin{proof}
Choose $e\in E_G$. Since, by \cite[Proposition~2.1]{NeVPVi20}, $K[E_G]/I(X_G)$ is Cohen--Macaulay
and $t_e^2$ is regular in $K[E_G]/I(X_G)$ and since $\AG= K[E_G]/(I(X_G),t_e^2)$, we deduce that $\AG$
is Gorenstein if and only if $K[E_G]/I(X_G)$ is (\emph{cf}.~\cite[Proposition~A.6.2]{HeHi11}).
\end{proof}


\section{The case $G=\K_n$}\label{sec: complete}

Below, $\K_n$ denotes a complete graph on 
$n$ vertices. The aim of this section is to compute 
the $h$-vector $(h_0,h_1,\dots,h_r)$ and the socle degrees of $\AG$, when $G=\K_n$. We begin 
by computing all possible values of $\tau(T,p)$, where $T$ is an even 
cardinality set of vertices and $p\in \ZZ_2$.

\begin{lemma}\label{lem: tau for complete}
If $G=\K_n$, with $n\geq 3$, and $\emptyset \not = T\subseteq V_G$ has even cardinality, then 
$$
\ts \{\tau(T,0),\tau(T,1)\}=\{\frac{|T|}{2},\frac{|T|}{2}+1\}.
$$
Additionally, $\tau(\emptyset,0)=0$ and $\tau(\emptyset,1)=3$.
\end{lemma}

\begin{proof}
If $T=\emptyset$ then, clearly, $\tau(\emptyset,0)=0$. Additionally, 
any set of three edges forming a triangle is an $\emptyset$-join of odd cardinality. Since 
there are no $(\emptyset,1)$-joins of smaller cardinality, 
$\tau(\emptyset,1)=3$. If $T=\{v,w\}$ with $v\not = w \in V_G$, then the edge $\set{v,w}$
is a minimum cardinality $(T,1)$-join and, taking any vertex $u\in V_G\setminus\set{v,w}$
the two edges $\{v,u\}$ and $\{w,u\}$ form a $(T,0)$-join of minimum cardinality. 
Hence $\tau(T,0)=\nicefrac{|T|}{2}+1 = 2$ and $\tau(T,1)=\nicefrac{|T|}{2}=1$. 
\smallskip

Let us now assume that $|T|\geq 4$ and let us take $J\subseteq E_G$, any matching of the vertices in $T$.
Then $|J|$ is equal to $\nicefrac{|T|}{2}$ and $J$ is a $T$-join of minimum cardinality.
Without loss of generality, assume that $|J|$ is even, 
so that $\tau(T,0)=\nicefrac{|T|}{2}$. Since $J$ has minimum cardinality among all $T$-joins, any $(T,1)$-join 
has cardinality at least $\nicefrac{|T|}{2}+1$. Let us show that one such $(T,1)$-join exists.
Take $\{v,w\} \in J$ and $u\in T\setminus \set{v,w}$. Then 
$$
(J\setminus \set{\set{v,w}})\cup \set{\set{u,v},\set{u,w}}
$$
is a $(T,1)$-join of cardinality $\nicefrac{|T|}{2}+1$.
\end{proof}

\begin{theorem}\label{thm: socle degrees for complete}
Let $G=\K_n$, with $n\geq 4$.
\begin{enumerate}\renewcommand{\itemsep}{.2cm}
\item $h_1=\binom{n}{2}$, $h_3 = 1+\binom{n}{4}+\binom{n}{6}$
and, if $i\not = 1,3$, $h_i=\binom{n}{2i-2}+\binom{n}{2i}$.
\item If $n$ is odd, the socle pairs are
$(\emptyset,1)$ together with $(T,\frac{n+1}{2}+2\ZZ)$, where 
$T\subseteq V_G$ is any subset of cardinality $n-1$. 
In particular, $\AG$ is of type $n+1$ and the sequence of socle degrees 
is $3$, $\frac{n+1}{2},\dots,\frac{n+1}{2}$.
\item If $n$ is even, the socle pairs are 
$\{(\emptyset,1), (V_G,0), (V_G,1)\}$. 
In particular, $\AG$ is of type $3$ and the sequence of socle degrees is $3, \frac{n}{2}, \frac{n}{2}+1$. 
\end{enumerate}

\end{theorem}

\begin{proof}
(i) If $i\not = 0,1,3$, by Lemma~\ref{lem: tau for complete} and Corollary~\ref{cor: Hilbert series and reduced parity joins}, 
$h_i$ is equal to the number of $T\subseteq V_G$ of cardinality 
$2i-2$ plus those of cardinality $2i$, hence $h_i=\binom{n}{2i-2}+\binom{n}{2i}$. 
This lemma gives $h_0=1$, which also agrees with this formula. 
The cases $i=1,3$ are dealt with similarly.  
\medskip 

\noindent
(ii) The statements about the type of $\AG$ and the socle degrees follow from 
Theorem~\ref{thm: characterization of socle}. 
By Lemma~\ref{lem: tau for complete}, it is clear that $(\emptyset,1)$ is a socle pair (as well if $n$ is even). 
Let $T\subseteq V_G$ be such that $|T|=n-1$. Then, by Lemma~\ref{lem: tau for complete},
\begin{equation}\label{eq: L785}
\ts \{\tau(T,0),\tau(T,1)\}=\{\frac{n-1}{2},\frac{n+1}{2}\}.
\end{equation}
Set $p=\frac{n+1}{2}+2\ZZ$, so that $\tau(T,p) = \frac{n+1}{2}$. If $e\in E_G$ is any edge, 
then, since $n\geq 4$, $T\sd e \not = \emptyset$ and hence
$$
\ts \{\tau(T\sd e,p),\tau(T\sd e,p+1)\}\subseteq \{\frac{n-2}{2},\frac{n-1}{2},\frac{n}{2},\frac{n+1}{2}\}.
$$
Since $\frac{n+3}{2}$ is does not belong to the set on the right of the above inclusion, 
by Lemma~\ref{lem: effect of adding edge}, we must have 
$$
\ts \tau(T\sd e,p+1) = \frac{n-1}{2} = \frac{n+1}{2}-1 = \tau(T,p)-1
$$
and we conclude that $(T,p)$ is a socle pair. It remains to show
that, besides those listed, no other pair $(T,p)$ is a socle pair. 
Let $T\subseteq V_G$ have even cardinality and let $p\in \ZZ_2$. Assume that $(T,p)\not = (\emptyset,1)$ and that there exist 
$v\not = w \in V_G\setminus T$. If $e=\set{v,w}$, then
$T\sd e\not = \emptyset$ and thus 
$$
\renewcommand{\arraystretch}{1.4}
\left \{
\begin{array}{l}
\tau(T,p) \in \{\frac{|T|}{2},\frac{|T|}{2}+1 \} \quad \text{and}\\ 
\tau(T\sd e,p+1) \in \{\frac{|T|}{2}+1,\frac{|T|}{2}+2 \}
\end{array}
\right.
$$
so that $\tau(T\sd e, p+1) = \tau(T,p)+1$, i.e., $(T,p)$ is not a socle pair. Finally if 
$|T|=n-1$ and $p=\frac{n-1}{2}+2\ZZ$, in which case, by \eqref{eq: L785},
$\tau(T,p) = \frac{n-1}{2}$, let us choose $e\in E_G$, an edge incident to the vertex in $V_G\setminus T$. Then 
$|T\sd e| = |T|$ and therefore 
$$
\ts \tau(T\sd e,p+1)=\frac{n+1}{2}=\tau(T,p)+1,
$$
i.e., $(T,p)$ is not a socle pair.
\medskip 

\noindent
(iii) As already observed, $(\emptyset,1)$ is a socle pair. 
Let $e\in E_G$. Since,
by Lemma~\ref{lem: tau for complete},
$$
\renewcommand{\arraystretch}{1.4}
\left \{
\begin{array}{l}
\ts \tau(V_G,p) \in \set{\frac{n}{2},\frac{n}{2}+1} \quad  \text{and}\\
\tau(V_G\sd e , p+1) \in \set{\frac{n}{2}-1,\frac{n}{2}},
\end{array}
\right.
$$
we cannot have $\tau(V_G\sd e,p+1)=\tau(V_G,p)+1$. Hence, using 
Lemma~\ref{lem: effect of adding edge}, we deduce that both $(V_G,0)$ and $(V_G,1)$ are socle pairs.
If $T\subsetneq V_G$ is a set of vertices of even cardinality, then there exist $v\not = w \in V_G\setminus T$.
Therefore, arguing as above, we can exclude $(T,p)$, for any $p\in \ZZ_2$.
\end{proof}

\begin{rmks}
(i) If $G=\K_n$ with $n=2k$, we observe that $(h_1,\dots,h_r)$, the 
$h$-vector of $\AG$, is nearly symmetric. 
Indeed if $i\not\in \set{1,3,k-2,k}$, then
$$
\ts h_i=\binom{2k}{2i-2}+\binom{2k}{2i} = \binom{2k}{2k+2-2i} + \binom{2k}{2k-2i} = h_{r-i},
$$ 
where $r=k+1$ is the maximum socle degree, while 
$$
\renewcommand{\arraystretch}{1.4}
\left \{
\begin{array}{l}
h_1 = \binom{2k}{2},\quad h_k=\binom{2k}{2k-2}+\binom{2k}{2k} = h_1 +1\quad \text{and} \\
h_3 = 1+\binom{2k}{4}+\binom{2k}{6},\quad h_{k-2}=\binom{2k}{2k-6}+\binom{2k}{2k-4} = h_3-1.
\end{array}
\right.
$$
(ii) If $n=5$ the socle degrees are $3,3,3,3,3,3$, i.e., $\AG$ is level. This is the only 
value of $n$ for which $\AG$, with $G=\K_n$, is level. 
\end{rmks}

\begin{example}\label{exa: maximal reduced parity join}
Let us use $G=\K_6$ to give an example of a reduced parity join that is maximal 
(under the inclusion order) in the set of reduced 
parity joins but not in the set of parity joins. Fix an ordering of $E_G$. By Theorems~\ref{thm: socle degrees for complete} 
and \ref{thm: characterization of socle} the only reduced parity joins that are maximal in the set
of all parity joins have cardinalities $\tau(\emptyset,1)=3$, $\tau(V_G,0)=4$ and $\tau(V_G,1)=3$.
The union of these $3$ reduced parity joins has cardinality at most $3+4+3=10$. Let 
$e\in E_G$ be an edge that does not belong to this union. Then, trivially, 
$\{e\}$ is a reduced parity join. Let $J$ be a reduced parity join, maximal in the set of all reduced parity 
joins, that contains $e$. Then $J$ cannot be maximal in the set of parity joins.
\end{example}


\section{The case $G=\K_{a,b}$}\label{sec: complete bipartite}

Let us denote by $\K_{a,b}$ a complete bipartite graph, with a partition, 
$V_G=V_1\cup V_2$, of its vertex set into parts of cardinalities $a$ and 
$b$, respectively. We will compute 
the $h$-vector $(h_0,h_1,\dots,h_r)$ and the socle degrees of $\AG$, when $G=\K_{a,b}$ by 
following a similar path as in previous section, except that, 
because $G$ is bipartite, we may omit the parity in the notation $\tau(T,p)$.

\begin{lemma}\label{lem: tau for complete bipartite}
If $G=\K_{a,b}$, with $a,b\geq 2$, and if $T\subseteq V_G$ has even cardinality, then, 
setting $T_1=T\cap V_1$ and $T_2=T\cap V_2$, $\ts \tau(T) = \max\{|T_1|,|T_2|\}$.
\end{lemma}

\begin{proof}
Without loss of generality, let 
us assume that $|T_1|\leq |T_2|$  
and let us take any a subset of $T_2'\subseteq T_2$
of cardinality $|T_1|$. Then, as $|T|$ is even and $|T|=2|T_1|+|T_2\setminus T_2'|$,
we must have $|T_2\setminus T_2'|$ even.
If $T_2'=\emptyset$, let $J=\emptyset \subseteq E_G$; 
otherwise, let us take $J\subseteq E_G$, a matching of the vertices in $T_1\cup T'_2$.
If $T_2\setminus T'_2=\emptyset$, let us take $L=\emptyset \subseteq E_G$; 
otherwise take any vertex $v\in V_1$ and let $L$ be the set of edges incident to 
$v$ and to a vertex of $T_2\setminus T_2'$. 
Then, $J$ is a $(T_1\cup T'_2)$-join and $L$ is a $(T_2\setminus T'_2)$-join. 
By Lemma~\ref{lem: symetric difference of Ti-joins}, 
$J\sd L = J\cup L$ is a $T$-join. Since $J\cup L$ has cardinality 
$$
|T_2'|+|T_2\setminus T_2'| = |T_2| 
$$
and any $T$-join has at least $|T_2|$ edges, we conclude that $\tau(T)=|T_2|$.
\end{proof}

\begin{theorem}
Let $G=\K_{a,b}$ be a complete bipartite graph. 
\begin{enumerate}\renewcommand{\itemsep}{.2cm}
\item $\ts h_i = 
\bigl [\sum_{k\geq 0}\binom{a}{i-2k}\bigr ]  \bigl [\sum_{k\geq 0}\binom{b}{i-2k}\bigr ] 
- \bigl [\sum_{k\geq 1}\binom{a}{i-2k}\bigr ]  \bigl [\sum_{k\geq 1}\binom{b}{i-2k}\bigr ]$.
\item $T\subseteq V_G$, of even cardinality,
is a socle set if and only if 
$$|T_i|=\max\{|T_1|,|T_2|\} \implies T_i =V_i.$$
In particular, if $a< b$, the socle degrees are 
$$
\underbrace{a, \cdots, a}_{\sum\limits_{i\geq 1} \binom{b}{a-2i}},\underbrace{b, \cdots, b}_{\sum\limits_{i\geq 1} \binom{a}{b-2i}}
$$
and if $a=b$, then $\AG$ is level, of socle degree $a$ and type $2^a-1$.

\end{enumerate}

\end{theorem}

\begin{proof}
(i) By Lemma~\ref{lem: tau for complete bipartite} and Corollary~\ref{cor: Hilbert series and reduced parity joins}, 
$h_i$ is the number of even cardinality subsets $T\subseteq V_G$
such that, if $T_1=T\cap V_1$ and $T_2=T\cap V_2$, $i=\max\{|T_1|,|T_2|\}$. 
The number of even cardinality $T\subset V_G$ such that $i=|T_1|>|T_2|$ is: 
$$
\ts \binom{a}{i}\sum_{k\geq 1} \binom{b}{i-2k}.
$$
Using symmetry and adding the case $|T_1|=|T_2|$ we deduce that 
$$
\renewcommand{\arraystretch}{1.4}
\begin{array}{rcl}
h_i & = & \binom{a}{i}\sum_{k\geq 1} \binom{b}{i-2k} + \binom{b}{i}\sum_{k\geq 1} \binom{a}{i-2k} + \binom{a}{i}\binom{b}{i}\\
 & = & \bigl [\binom{a}{i}+\sum_{k\geq 1}\binom{a}{i-2k}\bigr ]  \bigl [\binom{b}{i}+\sum_{k\geq 1}\binom{b}{i-2k}\bigr ] 
- \bigl [\sum_{k\geq 1}\binom{a}{i-2k}\bigr ]  \bigl [\sum_{k\geq 1}\binom{b}{i-2k}\bigr ]\\
 & = & \bigl [\sum_{k\geq 0}\binom{a}{i-2k}\bigr ]  \bigl [\sum_{k\geq 0}\binom{b}{i-2k}\bigr ] 
- \bigl [\sum_{k\geq 1}\binom{a}{i-2k}\bigr ]  \bigl [\sum_{k\geq 1}\binom{b}{i-2k}\bigr ].
\end{array}
$$

\noindent
(ii) Let us start by showing that $T$ is not a socle set if and only if there exists $i\in \{1,2\}$ such that 
\begin{equation}\label{eq: L841}
|T_i|=\max\{|T_1|,|T_2|\}\quad\text{and} \quad T_i\not = V_i.
\end{equation}
If $T$ is not a socle set, then
there exists $e\in E_G$ such that 
\begin{equation}\label{eq: L843}
\tau(T\sd e) = \tau(T)+1 = \max\{|T_1|,|T_2|\}+1.
\end{equation}
(\emph{cf}.~Lemmas~\ref{lem: effect of adding edge} and \ref{lem: tau for complete bipartite}).
Hence, for some $i\in \{1,2\}$ we have $(V_i\cap e)\cap T_i = \emptyset$ and 
$|T_i|=\max\{|T_1|,|T_2|\}$, which gives \eqref{eq: L841}. 
Conversely, suppose that \eqref{eq: L841} holds, for some $i\in \{1,2\}$.
Choose $v\in V_i\setminus T_i$ and $e\in E_G$, an edge incident to $v$. Then
\eqref{eq: L843} holds and therefore $T$ is not a socle set. 
\smallskip 

\noindent
Now, if $a\not = b$, then, by the above characterization, 
$$
\ts \sum_{i\geq 1} \binom{b}{a-2i} \quad \text{and}\quad \sum_{i\geq 1} \binom{a}{b-2i}
$$
are, respectively, the number of socle sets yielding socle degrees equal to $a$  
and the number of those yielding socle degrees equal to $b$.  
If $a=b$, the case of $T_1=V_1$, $T_2=V_2$ is not included in the previous 
counting, hence, in this case, the type of $\AG$ is then
$$
\ts 2\sum_{i\geq 1} \binom{a}{a-2i}+1 = 2(2^{a-1}-1)+1 = 2^a-1. \qedhere
$$
\end{proof}

\end{document}